\documentclass[12pt,a4paper]{amsart} %
\usepackage[
  a4paper,
  top=2.8cm,
  bottom=3.5cm,
  left=2.2cm,
  right=2.2cm,
  heightrounded
]{geometry}
\usepackage{amsmath,amsfonts,amssymb,amsthm}
\usepackage{comment}
\usepackage{todonotes}

\usepackage[colorlinks=true,linkcolor=blue,citecolor=blue,urlcolor=blue]{hyperref}
\usepackage[skip=0.8em plus 0.2em minus 0.1em]{parskip}
\allowdisplaybreaks

\theoremstyle{plain}
\newtheorem{question}{Question}

\newtheorem{theorem}{Theorem}

\newtheorem{corollary}{Corollary}
\newtheorem{lemma}{Lemma} 

\newtheorem{theoremx}{Theorem}

\newtheorem{definitionx}{Definition}

\newtheorem{lemmax}{Lemma}

\theoremstyle{remark}
\newtheorem{remark}{Remark}

\title{Discrete Pauli pairs}
\author{Torgeir Keun Lysen}
\email{torgeir.lysen@gmail.com}

\begin{document}
\begin{abstract}
    We determine sharp density thresholds for when discrete Pauli pairs with Gaussian decay must be classical Pauli pairs. More precisely, we identify when equality of the sampled moduli of two functions and their Fourier transforms forces equality of the global moduli in time and frequency, thereby answering a question of Ramos and Sousa. We also determine the sharp threshold for when discrete Pauli pairs must be weak Pauli pairs, where either the time-side or the frequency-side moduli agree identically. These can both be seen as phaseless versions of the results of Kulikov, Nazarov, and Sodin on Fourier uniqueness pairs.
\end{abstract}
\maketitle

\section{Introduction}
The Pauli problem asks to what extent a function is determined by the modulus of the function and the modulus of its Fourier transform. This problem was motivated by problems in quantum mechanics~\cite{pauli1933wellenmechanik}. More broadly, it is an instance of phase retrieval, where one seeks to reconstruct an object from phaseless intensity measurements. This kind of problem has gained substantial interest due to its applications in radiography~\cite{allman}, tomography~\cite{seelamantula}, X-ray crystallography~\cite{Elser,millane,Thibault}, microscopy and diffraction imaging~\cite{Drenth,Miao,Shechtman}, quantum mechanics~\cite{ORLOWSKI,Raymer}, data analysis~\cite{Mallat}, and other fields. For a mathematical survey of the topic see~\cite{grohs2020}.

In this paper we study a sampled version of the problem: when is equality of these two moduli on discrete sets enough to force equality of the global moduli? The problem is motivated by a question of W. Pauli, where the probability density of a particle was known only at countably many spatial points at two different times. Note that it is impossible for us to recover the phase under these conditions as there are infinitely many functions satisfying $|f(x)|\equiv|g(x)|$ and $|\hat{f}(\xi)|\equiv|\hat{g}(\xi)|$ without having $f(x)\equiv e^{i\vartheta}g(x)$ for some $\vartheta\in\mathbb{R}$. We use the following terminology to distinguish the global problem from its sampled analogue.
\begin{definitionx}[Pauli pair, weak Pauli pair, and discrete Pauli pair] Let $f,g:\mathbb{R}\to\mathbb{C}$.
    \begin{enumerate}
        \item[(i)] We say that $f$ and $g$ form a \emph{Pauli pair} if both $|f(x)|=|g(x)|$ almost everywhere and $|\hat f(\xi)|=|\hat g(\xi)|$ almost everywhere.
        \item[(ii)] We say that $f$ and $g$ form a \emph{weak Pauli pair} if $|f(x)|=|g(x)|$ almost everywhere or $|\hat f(\xi)|=|\hat g(\xi)|$ almost everywhere.
        \item[(iii)] We say that $f$ and $g$ form a \emph{discrete Pauli pair} for a pair of discrete sets $\Lambda,M\subset\mathbb{R}$ if $|f(\lambda)|=|g(\lambda)|$ for all $\lambda\in\Lambda$ and $|\hat{f}(\mu)|=|\hat{g}(\mu)|$ for all $\mu\in M$.
    \end{enumerate}
\end{definitionx}
We use the following normalization of the Fourier transform:
\[
\hat{f}(\xi):=\int_\mathbb{R}f(x)e^{-2\pi ix\xi}\,dx.
\]
The sampled Pauli problem can be seen as a phaseless version of the Fourier uniqueness problem, which has been studied extensively~\cite{adve2023density,kulikov2025,lysen2025criticalasymmetricfourieruniqueness,ramos2022fourier} since the work of Radchenko and Viazovska on Fourier interpolation~\cite{Radchenko2019}. The study of Fourier uniqueness pairs can be viewed as a discrete form of the uncertainty principle.
\begin{definitionx}[Fourier uniqueness pair]\label{def:fourier-uniqueness-pair}
    We call a pair $(\Lambda, M)$ with $\Lambda,M\subset \mathbb{R}$ a \emph{Fourier uniqueness pair} for a space $X\subset L^1(\mathbb{R})\cap L^2(\mathbb{R})$ if, for every $f\in X$,
    \[
    f|_\Lambda=0\quad\text{and}\quad \hat{f}|_M=0\quad\Longrightarrow\quad f\equiv 0.
    \]
    A pair that is not a \emph{Fourier uniqueness pair} is called a \emph{Fourier non-uniqueness pair}.
\end{definitionx}
We particularly note a result of Kulikov, Nazarov, and Sodin, who identified a broad class of Fourier uniqueness and non-uniqueness pairs called supercritical and subcritical pairs~\cite{kulikov2025}.

\begin{definitionx}[Supercritical and subcritical pairs]\label{def:density-criticality}
Given $1<p,q<\infty$ with $\frac{1}{p}+\frac{1}{q}=1$, we call a pair $(\Lambda, M)$ \emph{supercritical} if
\begin{align*}
        \limsup_{j\to\pm\infty}|\lambda_j|^{p-1}(\lambda_{j+1}-\lambda_j)\leq\overline{\alpha}\quad\text{and}\quad\limsup_{j\to\pm\infty}|\mu_j|^{q-1}(\mu_{j+1}-\mu_j)\leq \overline{\beta}
\end{align*}
with $\overline{\alpha}^\frac{1}{p}\overline{\beta}^\frac{1}{q}< \frac{1}{2}$. Conversely, the pair is \emph{subcritical} if there exists $\underline{\alpha}^\frac{1}{p}\underline{\beta}^\frac{1}{q}> \frac{1}{2}$ such that
\begin{align*}
        \liminf_{j\to\pm\infty}|\lambda_j|^{p-1}(\lambda_{j+1}-\lambda_j)\geq\underline{\alpha}\quad\text{and}\quad\liminf_{j\to\pm\infty}|\mu_j|^{q-1}(\mu_{j+1}-\mu_j)\geq \underline{\beta}.
\end{align*}
\end{definitionx}
Unless stated otherwise, we assume that the sequences $\Lambda,M\subset\mathbb{R}$ are ordered and unbounded in both directions.

We use $\mathcal{S}$ to denote the Schwartz space and $\mathcal{H}$ to denote the Fourier-symmetric Sobolev space
\[
\mathcal{H}:=\left\{f\in L^2(\mathbb{R}): \int_\mathbb{R}(1+x^2)|f(x)|^2\,dx+\int_\mathbb{R}(1+\xi^2)|\hat{f}(\xi)|^2\,d\xi<\infty\right\}.
\]
\begin{theoremx}[Kulikov--Nazarov--Sodin~\cite{kulikov2025}]\label{thm:kns-uniqueness}
    \textit{Suppose that $1 < p, q < \infty$, $\tfrac{1}{p} + \tfrac{1}{q} = 1$. Then}
\begin{itemize}
    \item[(i)] \textit{any supercritical pair $(\Lambda, M)$ is a uniqueness pair for $\mathcal{H}$;}
    \item[(ii)] \textit{any subcritical pair $(\Lambda, M)$ is a non-uniqueness pair for $\mathcal{S}$.}
\end{itemize}
\end{theoremx}
The connection with Fourier uniqueness enters the sampled Pauli problem through the following result of Ramos and Sousa.

\begin{theoremx}[Theorem~2(I) in~\cite{ramos2025pauli}]\label{thm:ramos-pauli}
     Fix any $A>0$. There exists a $\mathfrak{c}_A>0$ such that the following holds. Assume that $(\Lambda,M)$ satisfies
        \begin{equation*}
            \max\left\{\limsup_{j\to\pm\infty}|\lambda_j|(\lambda_{j+1}-\lambda_j),\,\limsup_{j\to\pm\infty}|\mu_j|(\mu_{j+1}-\mu_j)\right\}<\frac{1}{\mathfrak{c}_A}.
        \end{equation*}
        Let $f,g\in\mathcal{H}$ and assume that $g$ satisfies
        \[
        \sup_{x\in\mathbb{R}}|g(x)|e^{A\pi x^2}+\sup_{\xi\in\mathbb{R}}|\hat{g}(\xi)|e^{A\pi\xi^2}<\infty.
        \]
        Then $f$ and $g$ form a Pauli pair if and only if they form a discrete Pauli pair for $(\Lambda,M)$.
\end{theoremx}
For $A>1$, Hardy's uncertainty principle and Theorem~\ref{thm:kns-uniqueness} give $\mathfrak{c}_A=2$. For $A\leq1$, the methods in~\cite{ramos2025pauli} give $\mathfrak{c}_A=\frac{4\pi}{A}$ even under the stronger assumption that $f$ satisfies the same decay condition as $g$. They also point out that this is unlikely to be optimal, since one expects $\mathfrak{c}_A\to2^+$ as $A\to1^-$. In light of this, the following question arises.
\begin{question}[Question 23 in~\cite{ramos2025pauli}]\label{ques:ramos-sousa-23}
    What is the best $\mathfrak{c}_A$ such that Theorem~\ref{thm:ramos-pauli} holds?
\end{question}

Due to the following decay transfer result, we can assume that $f$ has roughly the same decay as $g$. This gives Theorem~\ref{thm:ramos-pauli} with $\mathfrak{c}_A=\frac{4\pi}{A}$ for $A\leq 1$ by the method in~\cite{ramos2025pauli}. However, as we will see, this is not optimal. 
\begin{theoremx}[Theorem 1 in~\cite{lysen2026discretehardyuncertaintyprinciple}]\label{thm:discrete-beurling}
    Let $a,b>0$, and let $\Lambda,M\subset\mathbb{R}$ be a supercritical pair with parameters $(p,q)$. Assume that
    \begin{equation*}
        \limsup_{j\to\pm\infty}|\lambda_j|^{p-1}(\lambda_{j+1}-\lambda_j)<\frac{1}{2}\left(\frac{b}{a}\right)^\frac{1}{q}\quad\text{and}\quad\limsup_{j\to\pm\infty}|\mu_j|^{q-1}(\mu_{j+1}-\mu_j)<\frac{1}{2}\left(\frac{a}{b}\right)^\frac{1}{p}.
    \end{equation*}
    If $f\in\mathcal{H}$ and
    \[
    \sup_{\lambda\in\Lambda}(1+|\lambda|)^{-K}|f(\lambda)|e^{a\pi|\lambda|^p}+\sup_{\mu\in M}(1+|\mu|)^{-K}|\hat{f}(\mu)|e^{b\pi|\mu|^q}<\infty
    \]
    for some $K>0$, then there exists a $\tilde{K}>0$ such that
    \[
    \sup_{x\in\mathbb{R}}|f(x)|(1+|x|)^{-\tilde{K}}e^{a\pi|x|^p}+\sup_{\xi\in\mathbb{R}}|\hat{f}(\xi)|(1+|\xi|)^{-\tilde{K}}e^{b\pi|\xi|^q}<\infty.
    \]
\end{theoremx}
The purpose of this paper is to give optimal thresholds for Question~\ref{ques:ramos-sousa-23}.

\section{Main results}
We first explain why Gaussian decay is essential. Without suitable decay conditions, no density condition on the sampling sets can force a discrete Pauli pair to be a Pauli pair.
\begin{theoremx}[Theorem~1 in~\cite{ramos2025pauli}]\label{thm:pauli-counterexamples-ramos}
    Let $\Lambda,M\subset\mathbb{R}$ be a pair of discrete sets. Then there exists a pair of functions $f,g\in\mathcal{S}$ with $|f(x)|\not\equiv |g(x)|$ and $|\hat{f}(\xi)|\not\equiv |\hat{g}(\xi)|$, such that $f$ and $g$ form a discrete Pauli pair.
\end{theoremx}
We provide a proof of this result, since the proof is short and illustrates how we will construct counterexamples and why the decay conditions are needed.
\begin{proof}
    Since $\Lambda$ and $M$ are discrete, we know that there exist non-empty intervals $I,J\subset\mathbb{R}$ such that $I\subset \mathbb{R}\backslash\Lambda$ and $J\subset\mathbb{R}\backslash M$. Choose non-trivial functions $\Phi\in C^\infty_c(I)$ and $\displaystyle{\hat{\Psi}\in C^\infty_c(J)}$. After replacing $\Psi$ by $e^{i\vartheta}\Psi$ for a suitable $\vartheta\in\mathbb{R}$, we may assume that $\Re(\Phi\overline{\Psi})\not\equiv0$ and $\Re(\hat{\Phi}\overline{\hat{\Psi}})\not\equiv0$. These functions are in the Schwartz space and satisfy $\Phi|_\Lambda=0$ and $\hat{\Psi}|_M=0$.
    
    Now let $f=\Phi+\Psi$ and $g=\Phi-\Psi$. These functions satisfy $|f(\lambda)|=|\Psi(\lambda)|=|g(\lambda)|$ for all $\lambda\in\Lambda$ and $|\hat{f}(\mu)|=|\hat{\Phi}(\mu)|=|\hat{g}(\mu)|$ for all $\mu\in M$. Since $\Re(\Phi\overline{\Psi})\not\equiv0$ and $\Re(\hat{\Phi}\overline{\hat{\Psi}})\not\equiv0$, they also satisfy $|f(x)|\not\equiv |g(x)|$ and $|\hat{f}(\xi)|\not\equiv|\hat{g}(\xi)|$.
\end{proof}
This obstruction comes from the freedom to choose compactly supported functions away from the sampling sets. To rule it out, we impose Gaussian-type decay in both time and frequency. We say that $f\in \mathcal{S}$ is in the Hardy class $E(a,b)$~\cite{Folland1997,GargThangavelu2010}, for some $a,b>0$, if
\[
\sup_{x\in\mathbb{R}}|f(x)|e^{a\pi|x|^2}+\sup_{\xi\in\mathbb{R}}|\hat{f}(\xi)|e^{b\pi|\xi|^2}<\infty.
\]

The first main result is a one-sided version of the sampled Pauli problem. It determines the sharp density threshold for forcing equality of moduli from samples of those moduli in only the time domain. We show that the optimal density parameter for this is
\begin{equation}\label{eq:pauli-constant}
    \mathfrak{c}_1(A)=
        \begin{cases}
        \frac{2}{A},\quad&A<\frac{1}{\sqrt{2}},\\
        4\sqrt{1-A^2},\quad&\frac{1}{\sqrt{2}}\leq A<1.
        \end{cases}
\end{equation}
\begin{theorem}\label{thm:pauli-sharpening}
    Fix any $0<A<1$ and let $\mathfrak{c}_1(A)$ be defined as in~\eqref{eq:pauli-constant}.
    \begin{enumerate}
        \item[(i)] Let $f,g\in E(A,A)$. Suppose $\Lambda\subset\mathbb{R}$ is a discrete set satisfying
    \begin{equation*}
        \limsup_{j\to\pm\infty}|\lambda_j|(\lambda_{j+1}-\lambda_j)<\frac{1}{\mathfrak{c}_1(A)}.
    \end{equation*}
     If $|f(\lambda)|=|g(\lambda)|$ for all $\lambda\in\Lambda$, then $|f(x)|\equiv|g(x)|$. 
    \item[(ii)] Conversely, if $\Lambda\subset\mathbb R$ is a discrete set and
    \[
    \liminf_{j\to\pm\infty}|\lambda_j|(\lambda_{j+1}-\lambda_j)>\frac{1}{\mathfrak{c}_1(A)},
    \]
    then there exist functions $f,g\in E(A,A)$ such that $|f(\lambda)|=|g(\lambda)|$ for all $\lambda\in\Lambda$ and $|f(x)|\not\equiv|g(x)|$.
    \end{enumerate}
\end{theorem}
Although this theorem only concerns the time-side modulus, applying it in both the time and frequency domains gives the sharp answer to Question~\ref{ques:ramos-sousa-23}. The additional threshold $2$ comes from the condition needed to transfer the sampled decay information for $f$ in Theorem~\ref{thm:discrete-beurling}.

\begin{corollary}\label{cor:pauli-sharpening}
    Fix any $0<A<1$ and let $\mathfrak{c}_1(A)$ be defined as in~\eqref{eq:pauli-constant}. Assume that $(\Lambda,M)$ satisfies
        \begin{equation}\label{eq:cor-pauli-sharpening-condition}
            \max\left\{\limsup_{j\to\pm\infty}|\lambda_j|(\lambda_{j+1}-\lambda_j),\,\limsup_{j\to\pm\infty}|\mu_j|(\mu_{j+1}-\mu_j)\right\}<\frac{1}{\max\{\mathfrak{c}_1(A),2\}}.
        \end{equation}
        Let $f\in\mathcal{H}$ and $g\in E(A,A)$. Then $f$ and $g$ form a Pauli pair if and only if they form a discrete Pauli pair for $(\Lambda,M)$.
\end{corollary}
\begin{proof}
    The forward implication is immediate. Conversely, assume that $f$ and $g$ form a discrete Pauli pair for $(\Lambda,M)$. The sampled decay of $g$ transfers to $f$ by Theorem~\ref{thm:discrete-beurling}, so $f\in E(A',A')$ for all $A'<A$. Since also $g\in E(A',A')$ for all $A'<A$, we can choose $A'$ such that~\eqref{eq:cor-pauli-sharpening-condition} holds with $\max\{\mathfrak{c}_1(A'),2\}^{-1}$ on the right-hand side. Applying Theorem~\ref{thm:pauli-sharpening} in both the time and frequency domains gives the desired result.
\end{proof}
There are two different notions of sharpness. The first asks whether the full Pauli pair conclusion follows. The second asks whether equality of both moduli can fail globally at the same time. These are not equivalent: a discrete Pauli pair may fail to be a Pauli pair because the time-side moduli differ, even if the Fourier-side moduli agree identically.

The next result shows sharpness for the Pauli threshold while keeping the Fourier-side modulus fixed in the range where this threshold is governed by the one-sided obstruction. The construction is similar to the one in Theorem~\ref{thm:pauli-sharpening}(ii), but it also enforces $|\hat{f}(\xi)|\equiv|\hat{g}(\xi)|$.
\begin{theorem}\label{thm:discrete-pauli-frequency-match-nonuniqueness}
    Fix any $0<A<\frac{\sqrt{3}}{2}$ and let $\mathfrak{c}_1(A)$ be defined as in~\eqref{eq:pauli-constant}. Assume that $\Lambda\subset[0,\infty)$ satisfies
    \[
    \liminf_{j\to\infty}|\lambda_j|(\lambda_{j+1}-\lambda_j)>\frac{1}{\max\{\mathfrak{c}_1(A),2\}}.
    \]
    Then there exists a pair of functions $f\in\mathcal{H}$ and $g\in E(A,A)$ such that $|\hat{f}(\xi)|=|\hat{g}(\xi)|$ for all $\xi\in\mathbb{R}$, $|f(\pm \lambda)|=|g(\pm\lambda)|$ for all $\lambda\in\Lambda$, and $|f(x)|\not\equiv|g(x)|$.
\end{theorem}
The symmetry in the sampling condition is used in the construction: the functions $\Phi$ and $\Psi$ are chosen to be even and odd so that their Fourier transforms are pointwise orthogonal.

Combining Corollary~\ref{cor:pauli-sharpening} and Theorem~\ref{thm:discrete-pauli-frequency-match-nonuniqueness}, we see that the optimal constant in Question~\ref{ques:ramos-sousa-23} is $\mathfrak{c}_A=\max\{\mathfrak{c}_1(A),2\}$.

We finally ask how dense $(\Lambda,M)$ can be while still allowing discrete Pauli pairs that are not weak Pauli pairs, meaning that both $|f(x)|\not\equiv|g(x)|$ and $|\hat{f}(\xi)|\not\equiv|\hat{g}(\xi)|$. The optimal density parameter for this problem is 
\begin{equation}\label{eq:pauli-nonclassical-constant}
\mathfrak{c}_2(A):=\begin{cases}
            \sqrt{\frac{(3-\sqrt{1-8A^2})^3}{2(1-\sqrt{1-8A^2})}},\quad &0<A<\frac{1}{3},\\
            4\sqrt{1-A^2},\quad &\frac{1}{3}\leq A<\frac{\sqrt{3}}{2},\\
            2,\quad&\frac{\sqrt{3}}{2}\leq A<1.
        \end{cases}
\end{equation}
\begin{theorem}\label{thm:pauli-non-uniqueness}
    Fix any $0<A<1$ and let $\mathfrak{c}_2(A)$ be defined as in~\eqref{eq:pauli-nonclassical-constant}. 
    \begin{enumerate}
    \item[(i)] Suppose $\Lambda,M\subset\mathbb{R}$ are two discrete sets such that
    \begin{equation*}
        \max\left\{\limsup_{j\to\pm\infty}|\lambda_j|(\lambda_{j+1}-\lambda_j),\,\limsup_{j\to\pm\infty}|\mu_j|(\mu_{j+1}-\mu_j)\right\}<\frac{1}{\mathfrak{c}_2(A)}.
    \end{equation*}
        Let $f\in\mathcal{H}$ and $g\in E(A,A)$. If $f$ and $g$ form a discrete Pauli pair for $(\Lambda,M)$, then they form a weak Pauli pair.
    \item[(ii)] Conversely, if $\Lambda,M\subset\mathbb R$ are discrete sets and
    \begin{equation*}
        \min\left\{\liminf_{j\to\pm\infty}|\lambda_j|(\lambda_{j+1}-\lambda_j),\,\liminf_{j\to\pm\infty}|\mu_j|(\mu_{j+1}-\mu_j)\right\}>\frac{1}{\mathfrak{c}_2(A)}
    \end{equation*}
    then there exist functions $f\in\mathcal{H}$ and $g\in E(A,A)$ forming a discrete Pauli pair for $(\Lambda, M)$ such that $f$ and $g$ do not form a weak Pauli pair.
     \end{enumerate}
\end{theorem}
\begin{remark}\label{rem:squared-pauli-data}
    If one assumes in Theorem~\ref{thm:pauli-non-uniqueness}(i) that $f(\lambda)^2=g(\lambda)^2$ for all $\lambda\in\Lambda$ and $\hat{f}(\mu)^2=\hat{g}(\mu)^2$ for all $\mu\in M$ instead of $f$ and $g$ merely being a discrete Pauli pair, then it follows that $f(x)^2\equiv g(x)^2$ and $\hat{f}(\xi)^2\equiv\hat{g}(\xi)^2$. Similarly, for part~(ii) of the theorem, one can construct discrete Pauli pairs satisfying $f(\lambda)^2=g(\lambda)^2$ for all $\lambda\in\Lambda$ and $\hat{f}(\mu)^2=\hat{g}(\mu)^2$ for all $\mu\in M$ where $|f(x)|\not\equiv |g(x)|$ and $|\hat{f}(\xi)|\not\equiv|\hat{g}(\xi)|$. Both of these results hold with the same density threshold $\mathfrak{c}_2(A)$. This means that under stronger conditions one can perform sign retrieval.
\end{remark}

\begin{remark}
    The thresholds $\max\{\mathfrak{c}_1(A),2\}$ and $\mathfrak{c}_2(A)$ agree when $\frac{1}{\sqrt{2}}\leq A<1$. In that range one can choose $\Phi$ and $\Psi$ to have the optimal density of zeros in both the time and frequency domain. In the range $0<A<\frac{1}{\sqrt{2}}$, one must choose to optimize the density of zeros in either the time domain or the frequency domain.
\end{remark}
We give a short overview of the proof strategy of Theorem~\ref{thm:pauli-sharpening} and Theorem~\ref{thm:pauli-non-uniqueness}. We first prove a sharp Fourier uniqueness theorem for functions in the Hardy class. The non-uniqueness part of that result can be used to construct functions $\Phi,\Psi\in E(A,A)$ with suitable zero sets in time and frequency. These functions can be used to construct counterexamples using the same strategy as in Theorem~\ref{thm:pauli-counterexamples-ramos}. This gives us Theorem~\ref{thm:pauli-sharpening}(ii) and Theorem~\ref{thm:pauli-non-uniqueness}(ii). By estimating the Phragm\'en--Lindel\"of indicator of $H(z):=f(z)\overline{f(\overline{z})}-g(z)\overline{g(\overline{z})}$ using the real zeros of $H$, we obtain a contradiction unless $H\equiv0$. This proves Theorem~\ref{thm:pauli-sharpening}(i). Similarly, we get the proof of Theorem~\ref{thm:pauli-non-uniqueness}(i) by also estimating the Phragm\'en--Lindel\"of indicator of $\tilde{H}(z):=\hat{f}(z)\overline{\hat{f}(\overline{z})}-\hat{g}(z)\overline{\hat{g}(\overline{z})}$.

\section{Fourier uniqueness in the Hardy class}\label{sec:gelfand-shilov-uniqueness}
The purpose of this section is to prove a sharp version of Theorem~\ref{thm:kns-uniqueness} for functions in the Hardy class.

To carry out these arguments, we need to impose some regularity condition on the zero sets of the functions. We define the counting function $n_\Gamma(r):=|\Gamma\cap r\mathbb{D}|$.
\begin{definitionx}[$p$-smooth sequences]
Let $p>0$ and let $\Gamma=(\gamma_j)$ be a sequence of points in $\mathbb{C}$ lying on a ray
$\arg(\gamma_j)=\theta$, $|\gamma_1|<|\gamma_2|<\cdots$, $|\gamma_j|\uparrow\infty$.
We call the sequence $\Gamma$ \emph{$p$-smooth} with a positive density
$D=D(\Gamma,p)$ (with respect to the exponent $p$) if
\begin{enumerate}
\item $\bigl|n_\Gamma(r)-Dr^p\bigr|=O(1), \qquad r\to\infty;$
\item there exists $d>0$ such that
\[
|\gamma_{j+1}-\gamma_j|\ge d\,(1+|\gamma_j|)^{1-p}, \qquad j=1,2,\ldots
\]
\end{enumerate}
\end{definitionx}
In the following results, we use the notation $\Gamma_+:=\Gamma\cap[0,\infty)$ and $\Gamma_-:=\Gamma\cap(-\infty,0]$ to denote the positive and negative parts of a real set.
\begin{theorem}\label{thm:uniqueness-gelfand-shilov}
    Let $a,b>0$ with $ab<1$, and let $(\Lambda, M)$ be a pair where $\Lambda_-,\Lambda_+,M_-,M_+$ are $2$-smooth sequences with densities $D_\Lambda=D(\Lambda_-,2)=D(\Lambda_+,2)$ and $D_M=D(M_-,2)=D(M_+,2)$. Assume that $f\in E(a,b)$.
    \begin{enumerate}
        \item[(i)] If
        \begin{align*}
                D_\Lambda>\sqrt{a\left(\frac{1}{b}-a\right)}\quad\text{and}\quad D_M>\sqrt{b\left(\frac{1}{a}-b\right)},
        \end{align*}
        then $f|_\Lambda=0$ and $\hat f|_M=0$ imply $f\equiv0$.
        \item[(ii)] Conversely, there exists an infinite-dimensional space of functions $f\in E(a,b)$ vanishing on sets $(\Lambda,M)$ with densities
\begin{align*}
        D_\Lambda<\sqrt{a\left(\frac{1}{b}-a\right)}\quad\text{and}\quad D_M<\sqrt{b\left(\frac{1}{a}-b\right)}.
\end{align*}
\end{enumerate}
\end{theorem}
We note that many of the same techniques can be used for general Hölder conjugates $(p,q)$, but the case $p,q=2$ has been chosen to simplify expressions and notation.

We also note the following corollary of the previous theorem using notation that more closely aligns with that of Theorem~\ref{thm:kns-uniqueness}.
\begin{corollary}
    Fix $0<A<1$.
    \begin{enumerate}
        \item[(i)] A pair $(\Lambda, M)$ is a Fourier uniqueness pair for the space $E(A,A)$ if
    \begin{align*}
            \max\left(\limsup_{j\to\pm\infty}|\lambda_j|(\lambda_{j+1}-\lambda_j),\limsup_{j\to\pm\infty}|\mu_j|(\mu_{j+1}-\mu_j)\right)<\frac{1}{2\sqrt{1-A^2}}.
    \end{align*}
    \item[(ii)] Conversely, the pair is a non-uniqueness pair for the space $E(A,A)$ if
    \begin{align*}
            \min\left(\liminf_{j\to\pm\infty}|\lambda_j|(\lambda_{j+1}-\lambda_j),\liminf_{j\to\pm\infty}|\mu_j|(\mu_{j+1}-\mu_j)\right)>\frac{1}{2\sqrt{1-A^2}}.
    \end{align*}
    \end{enumerate}
\end{corollary}

\subsection{Entire-function tools}\label{sec:entire-function-tools}
This subsection collects the entire-function tools used in the non-uniqueness arguments.

We begin with a special case of a theorem by Levin on the Phragm\'en--Lindel\"of indicator of canonical products where the zero sets have a given angular density~\cite[Chapter 2, Theorem 5]{levinentirefunctions}. We use the simplified version given in~\cite{radchenko2024perturbedlatticecrossesheisenberg,ramos2025pauli} which only considers canonical products of order $2$ with zeros on the real and imaginary line. A self-contained proof of this result can be found in the appendix of~\cite{kulikov2025}.

Let $k:[-\pi,\pi]\to\mathbb{R}$ be defined by
\begin{equation}\label{eq:k-definition}
    k(\theta):=\beta\pi\sin(2\theta)-\gamma\cos(2\theta),
\end{equation}
for $\theta\in[0,\pi/2]$. We extend $k$ to $[-\pi,\pi]$ so that it is even and symmetric with respect to $\pi/2$.

\begin{definitionx}[$k$-regular set]
Let $k$ be defined as above. We call a discrete set
\[
    \mathcal{Z}
    =\mathcal{Z}_1^+\cup\mathcal{Z}_1^-
     \cup\mathcal{Z}_2^+\cup\mathcal{Z}_2^-
    \subset\mathbb{C},
\]
where
\[
    \mathcal{Z}_j^\pm
    \subset
    \left\{z\in\mathbb{C}:\arg(z)=\frac{(j\pm1)\pi}{2}\right\},
\]
$k$-regular if:
\begin{enumerate}
    \item $\mathcal{Z}_j^\pm$ has density $m_j/4\pi$ with respect to exponent $2$. That is,
    \[
        \left|\mathcal{Z}_j^\pm\cap(0,re^{i\theta_j^\pm})\right|
        \sim \frac{r^2m_j}{4\pi}
    \]
    as $r\to\infty$, where $m_1=2\pi\beta$, $m_2=2\gamma$, $\theta_j^\pm=\frac{(j\pm1)\pi}{2}$.
    \item The disks
    \[
        \{D_z\}_{z\in\mathcal{Z}}
        =
        \left\{B\left(z,c(1+|z|)^{-1}\right)\right\}_{z\in\mathcal{Z}}
    \]
    are all disjoint for some $c>0$.
\end{enumerate}
\end{definitionx}

\begin{theoremx}\label{thm:levin-prescribed-zeros}
Let $k$ be as in~\eqref{eq:k-definition}. Then, for any $k$-regular set $\mathcal{Z}$, there exists an entire function $S$, whose zeros are simple and coincide with $\mathcal{Z}$, such that, for every $\varepsilon>0$,
\begin{align}
    |S(re^{i\theta})|
    &\le C_\varepsilon e^{(k(\theta)+\varepsilon)r^2},
    &&\text{everywhere in } \mathbb{C}, \label{eq:levin-upper} \\
    |S(re^{i\theta})|
    &\ge c_\varepsilon e^{(k(\theta)-\varepsilon)r^2},
    &&\text{whenever } re^{i\theta}\notin\bigcup_{z\in\mathcal{Z}}D_z.
    \label{eq:levin-lower}
\end{align}
\end{theoremx}
Combining \eqref{eq:levin-lower} with the minimum modulus principle, we have that, for $z\in\mathcal{Z}$,
\begin{equation}\label{eq:levin-derivative-lower}
    |S'(z)|\ge c_\varepsilon e^{(k(\arg z)-\varepsilon)|z|^2}.
\end{equation}

The next lemma explains why a generating function with this indicator has the required Fourier decay.
\begin{lemma}\label{lem:fourier-decay-bound}
    Let $f:\mathbb{C}\to\mathbb{C}$ be an entire function of order $2$ and $a,b>0$. Assume that 
    \[
    h_f(\theta)\leq m\pi|\sin(2\theta)|-a\pi\cos(2\theta)\quad\text{for all }\theta\in[0,2\pi)
    \]
    and $0\leq m<\sqrt{a\left(\frac{1}{b}-a\right)}$. Then $\max\{h_{\hat{f}}(0),h_{\hat{f}}(\pi)\}<-b\pi$.
\end{lemma}
\begin{proof}
We choose $\varepsilon\in(0,a/2)$ sufficiently small such that 
\[
m\leq \sqrt{(a-2\varepsilon)\left(\frac{1}{b+\varepsilon}-(a+2\varepsilon)\right)}.
\]
We now want to check the inequality
\begin{equation}\label{eq:fourier-decay-indicator-inequality}
    m\pi|\sin(2\theta)|-a\pi\cos(2\theta)+2\pi\varepsilon\leq\frac{\pi}{b+\varepsilon}\sin^2\theta\quad\text{for all }\theta\in\left[0,2\pi\right).
\end{equation}
Because of symmetry, we can assume without loss of generality that $\theta\in[0,\pi/2)$. Let $u=\tan\theta$; then $\cos(2\theta)=\frac{1-u^2}{1+u^2}$, $\sin 2\theta=\frac{2u}{1+u^2}$, and $\sin^2\theta=\frac{u^2}{1+u^2}$. We can now rewrite the inequality as
\[
m\pi\frac{2u}{1+u^2}-a\pi\frac{1-u^2}{1+u^2}+2\pi\varepsilon\leq\frac{\pi}{b+\varepsilon}\left(\frac{u^2}{1+u^2}\right).
\]
We can simplify this to $2um-a(1-u^2)+2\varepsilon(1+u^2)\leq\frac{u^2}{b+\varepsilon}$.
Rearranging, we see that
\[
2\varepsilon-a\leq u^2\left(\frac{1}{b+\varepsilon}-(a+2\varepsilon)\right)-2um.
\]
We notice that the right-hand side is minimized when $u=m\left(\frac{1}{b+\varepsilon}-(a+2\varepsilon)\right)^{-1}$. This gives us the expression
\[
2\varepsilon-a\leq m^2\left(\frac{1}{b+\varepsilon}-(a+2\varepsilon)\right)^{-1}-2m^2\left(\frac{1}{b+\varepsilon}-(a+2\varepsilon)\right)^{-1}
\]
which is equivalent to $m\leq \sqrt{(a-2\varepsilon)\left(\frac{1}{b+\varepsilon}-(a+2\varepsilon)\right)}$, hence~\eqref{eq:fourier-decay-indicator-inequality} holds.

We can upper bound the Fourier transform of $f$ using a contour shift and Hölder's inequality:
    \begin{align*}
|\hat{f}(\xi)|&=\left|\int_\mathbb{R}f(x+iy)e^{-2\pi i(x+iy)\xi}\,dx\right|\\
    &\leq \left(\int_\mathbb{R}e^{-\pi\varepsilon|x|^2}\,dx\right)\inf_{y\in\mathbb{R}}\sup_{x\in\mathbb{R}}|f(x+iy)|e^{\pi\varepsilon|x|^2-2\pi |\xi||y|}\\
    &\leq C_\varepsilon\inf_{y\in\mathbb{R}}\sup_{x\in\mathbb{R}}e^{(h_f(\theta)+\pi\varepsilon)|x+iy|^2+\pi\varepsilon|x|^2-2\pi|\xi| |y|}.
\end{align*}
Since $h_f(\theta)+2\pi\varepsilon\leq\frac{\pi}{b+\varepsilon}|\sin\theta|^2$, we know that
\[
(h_f(\theta)+\pi\varepsilon)|x+iy|^2+\pi\varepsilon|x|^2\leq \frac{\pi}{b+\varepsilon}|y|^2.
\]
Hence,
\[
C_\varepsilon \inf_{y\in\mathbb{R}}\sup_{x\in\mathbb{R}}e^{(h_f(\theta)+\pi\varepsilon)|x+iy|^2+\pi\varepsilon|x|^2-2\pi|\xi||y|}\leq C_\varepsilon \inf_{y\in\mathbb{R}} e^{\frac{\pi}{b+\varepsilon}|y|^2-2\pi|\xi||y|}.
\]
Completing the square, we see that 
\[
|\hat{f}(\xi)|\leq C_\varepsilon\inf_{y\in\mathbb{R}} e^{\frac{\pi}{b+\varepsilon}|y|^2-2\pi|\xi||y|}\leq C_\varepsilon e^{-(b+\varepsilon)\pi|\xi|^2}.\qedhere
\]
\end{proof}

\subsection{Weighted simultaneous interpolation}
The non-uniqueness proof also requires simultaneous interpolation in time and frequency. We use the following version of Lemma~7.4 in~\cite{kulikov2025}.

Let $\alpha:\Lambda\to\mathbb{C}$ and $\beta:M\to\mathbb{C}$. We write $\kappa=(\alpha,\beta)$.
We say that $\kappa$ is in the weighted $\ell^1$-space $\mathfrak{S}(a,b)$ if the norm
\[
\|\kappa\|_{\mathfrak{S}(a,b)}:=\sum_{\lambda\in\Lambda}|\alpha(\lambda)|e^{a\pi|\lambda|^2}+\sum_{\mu\in M}|\beta(\mu)|e^{b\pi|\mu|^2}
\]
is finite. 
\begin{lemma}\label{lem:simultaneous-weighted-interpolation}
    Let $(\Lambda, M)$ be a pair where $\Lambda_-,\Lambda_+,M_-,M_+$ are $2$-smooth sequences with densities $D_\Lambda=D(\Lambda_-,2)=D(\Lambda_+,2)$ and $D_M=D(M_-,2)=D(M_+,2)$. Furthermore, assume that
\begin{align*}
        D_\Lambda<\sqrt{a\left(\frac{1}{b}-a\right)}\quad\text{and}\quad D_M<\sqrt{b\left(\frac{1}{a}-b\right)}.
\end{align*}
Then there exists $L>0$ such that for every $\kappa\in\mathfrak{S}(a,b)$ there exists $f\in\mathcal{S}$ such that
\[
f(\lambda)=\alpha(\lambda)\quad\text{for all }\lambda\in\Lambda_L
\qquad\text{and}\qquad
\hat{f}(\mu)=\beta(\mu)\quad\text{for all }\mu\in M_L,
\]
where $\Lambda_L:=\Lambda\backslash[-L,L]$ and $M_L:=M\backslash[-L,L]$. Additionally, this $f$ satisfies 
\[
h_f(0),h_f(\pi)\leq -a\pi\quad\text{and}\quad h_{\hat{f}}(0),h_{\hat{f}}(\pi)\leq-b\pi.
\]
\end{lemma}
\begin{proof}
We define the function
\[
k(\theta):=D_\Lambda\pi\sin(2|\theta|)-a\pi\cos(2\theta),\quad\theta\in\left[-\frac{\pi}{2},\frac{\pi}{2}\right].
\]
We then extend $k$ to be even, symmetric around $\frac{\pi}{2}$, and $2\pi$-periodic. We choose a $k$-regular zero set whose real zeros contain $\Lambda$, and let $\Phi(z)$ be the function constructed from Theorem~\ref{thm:levin-prescribed-zeros} with indicator $k$ such that $\Phi|_\Lambda=0$. 

Now let
\[
\Phi_\lambda(z):=\frac{\Phi(z)}{\Phi'(\lambda)(z-\lambda)}
\]
This function satisfies 
\[
\Phi_\lambda(\lambda')=\begin{cases}
    1,\quad\lambda=\lambda'\\
    0,\quad\lambda\neq\lambda'.
\end{cases}
\]
By Lemma~\ref{lem:fourier-decay-bound}, there exists $\varepsilon>0$ such that $h_{\hat{\Phi}_\lambda}(0),h_{\hat{\Phi}_\lambda}(\pi)\leq-(b+\varepsilon)\pi$. We also know by the minimum modulus principle~\eqref{eq:levin-derivative-lower} that $|\Phi'(\lambda)|\geq \tilde{C}_\varepsilon e^{-a\pi|\lambda|^2}$ for some $\tilde{C}_\varepsilon>0$, hence
\[
|\Phi_\lambda(x)|\leq C_1(\varepsilon)e^{a\pi|\lambda|^2-a\pi |x|^2}\quad\text{and}\quad|\hat{\Phi}_\lambda(\xi)|\leq C_1(\varepsilon)e^{a\pi|\lambda|^2-(b+\varepsilon)\pi |\xi|^2}
\]
where $C_1(\varepsilon):=\frac{C_\varepsilon}{\tilde{C}_\varepsilon}$. Similarly, we can also construct a function $\Psi(z)$ such that 
\[
|\Psi_\mu(x)|\leq C_2(\varepsilon)e^{b\pi|\mu|^2-(a+\varepsilon)\pi |x|^2}\quad\text{and}\quad|\hat{\Psi}_\mu(\xi)|\leq C_2(\varepsilon)e^{b\pi|\mu|^2-b\pi |\xi|^2}
\]
and
\[
\hat{\Psi}_\mu(\mu')=\begin{cases}
    1,\quad\mu=\mu'\\
    0,\quad\mu\neq\mu'.
\end{cases}
\]
We now choose $L>0$ such that 
\[
\sum_{\lambda\in \Lambda_L} e^{-\varepsilon\pi|\lambda|^2}+\sum_{\mu\in M_L}e^{-\varepsilon\pi|\mu|^2}\leq\frac{1}{4C(\varepsilon)}.
\]
Here $C(\varepsilon):=\max\{C_1(\varepsilon),C_2(\varepsilon)\}$. We define the function
\[
f_1(x):=\sum_{\lambda\in\Lambda_L}\alpha_1(\lambda)\Phi_\lambda(x)+\sum_{\mu\in M_L}\beta_1(\mu)\Psi_\mu(x)
\]
where $\alpha_1(\lambda)=\alpha(\lambda)$ and $\beta_1(\mu)=\beta(\mu)$. We define recursively
\[
f_{j+1}(x):=\sum_{\lambda\in\Lambda_L}\alpha_{j+1}(\lambda)\Phi_\lambda(x)+\sum_{\mu\in M_L}\beta_{j+1}(\mu)\Psi_\mu(x)
\]
with $\alpha_{j+1}(\lambda):=\alpha_j(\lambda)-f_j(\lambda)$ and $\beta_{j+1}(\mu):=\beta_j(\mu)-\hat{f}_j(\mu)$.

Now let $F_n(x):=\sum_{j=1}^nf_j(x)$. We see that
$|F_n(\lambda)-\alpha(\lambda)|=|\alpha_{n+1}(\lambda)|$. We now have 
\begin{align*}
    \sum_{\lambda\in\Lambda_L}|\alpha_{n+1}(\lambda)|e^{a\pi|\lambda|^2}&=\sum_{\lambda\in\Lambda_L}\left|\sum_{\mu\in M_L}\beta_n(\mu)\Psi_\mu(\lambda)\right|e^{a\pi|\lambda|^2}\\
    &\leq \left(C(\varepsilon)\sum_{\lambda\in\Lambda_L}e^{-\varepsilon\pi|\lambda|^2}\right)\left(\sum_{\mu\in M_L}|\beta_n(\mu)|e^{b\pi|\mu|^2}\right).
\end{align*}
Similarly,
\begin{align*}
    \sum_{\mu\in M_L}|\beta_{n+1}(\mu)|e^{b\pi|\mu|^2}
    &\leq \left(C(\varepsilon)\sum_{\mu\in M_L}e^{-\varepsilon\pi|\mu|^2}\right)\left(\sum_{\lambda\in\Lambda_L}|\alpha_n(\lambda)|e^{a\pi|\lambda|^2}\right).
\end{align*}
Hence, we get the inequality 
\[
\|\kappa_{n+1}\|_{\mathfrak{S}(a,b)}=\sum_{\lambda\in\Lambda_L}|\alpha_{n+1}(\lambda)|e^{a\pi|\lambda|^2}+\sum_{\mu\in M_L}|\beta_{n+1}(\mu)|e^{b\pi|\mu|^2}\leq\frac{1}{2}\|\kappa_n\|_{\mathfrak{S}(a,b)}.
\]
Consequently, $F_n$ converges to the solution of the interpolation problem.
\end{proof}
\subsection{Indicator lower bounds from zeros}
We say that a function $h$ is $p$-trigonometrically convex in a sector $\theta_1<\theta<\theta_2$ with $0<\theta_2-\theta_1<\frac{\pi}{p}$ if 
\[
h(\theta)\leq\frac{\sin(p(\theta_2-\theta))}{\sin(p(\theta_2-\theta_1))}h(\theta_1)+\frac{\sin(p(\theta-\theta_1))}{\sin(p(\theta_2-\theta_1))}h(\theta_2).
\]
Crucially, we can estimate the size of the indicator function $h_f$ with the following theorem.
\begin{theoremx}[Theorem~1, Lecture 8 in~\cite{levinentirefunctions}]\label{thm:levin-indicator}
Let $f(z)$ be a holomorphic function inside an angle, and suppose that it satisfies the inequality $|f(re^{i\theta})|=O(e^{cr^p})$ for some $c>0$ in a sector $\theta_1<\theta<\theta_2$ with $0<\theta_2-\theta_1<\frac{\pi}{p}$. Then its indicator function $h_f$ with respect to the order $p$ is a $p$-trigonometrically convex function.
\end{theoremx}
Note that because of the trigonometric convexity, the Phragm\'en--Lindel\"of indicator of a non-zero function can never equal $-\infty$. Indeed, if $f(z)\not\equiv0$ is entire of finite order $p$, then its Phragm\'en--Lindel\"of indicator is a finite $p$-trigonometrically convex function. If $h_f(\theta_0)=-\infty$ for some $\theta_0$ and $h_f(\theta)$ is trigonometrically convex on an interval with $\theta_0$ as an endpoint, then this would force $h_f(\theta)=-\infty$ on a neighborhood of $\theta_0$, and then by iteration we could extend this in every direction. This is impossible for a non-zero entire function due to Liouville's theorem.

We will also need the following estimate of the real line decay of a function based on its zero density on the real line.
\begin{lemmax}[Lemma 3 in~\cite{lysen2026discretehardyuncertaintyprinciple}]\label{lem:jensen-counting}
    Suppose that $0<\varphi<\frac{\pi}{2}$ and that $F$ is analytic in $S(\varphi):=\{z\in\mathbb{C}:|\arg(z)|<\varphi\}$, continuous up to the boundary, and satisfies $|F(z)|\leq C_F e^{a\pi|\Im(z)|}$ with some $a>0$. Let $\Gamma=(\gamma_j)_{j\in\mathbb{Z}}\subset\mathbb{R}_+$ be a discrete set such that, for some $\delta>0$ and $C_\Gamma>0$,
    \[
    |\Gamma\cap[u,v]|\geq a(1+\delta)(v-u)-C_\Gamma,\quad 0\leq u<v.
    \]
    If $F|_\Gamma=0$, then 
    \[
    \limsup_{x\to\infty}\frac{\log|F(x)|}{x}\leq -2a\delta\sin\varphi<0.
    \]
\end{lemmax}
Combining the two previous lemmas gives a lower bound on the Phragm\'en--Lindel\"of indicator of a function with a $p$-smooth zero set on the real line.
\begin{lemma}\label{lem:indicator-lower-bound-from-zeros}
        Let $\Gamma\subset\mathbb{R}_+$ be a $p$-smooth sequence with density $D(\Gamma,p)$ where $p\geq 1$. If $f$ is an entire function of order $p$ and $f|_\Gamma=0$, then
        \[
D(\Gamma,p)\pi\sin(p\theta)+h_f(0)\cos(p\theta)\leq \max\{h_f(\theta),h_f(-\theta)\}, \quad\text{for all }\theta\in\left[0,\pi/p\right].
\]
\end{lemma}
\begin{proof}
    Assume that there exist $\delta>0$ and $0<\theta_0<\frac{\pi}{2p}$ such that
\begin{equation}\label{eq:indicator-lower-bound-contradiction}
    \frac{D(\Gamma,p)}{1+\delta}\pi\sin(p\theta_0)+h_f(0)\cos(p\theta_0)> \max\{h_f(\theta_0),h_f(-\theta_0)\}.
\end{equation}

Now define $F_\varepsilon(z)=e^{-(h_f(0)-\varepsilon)z}f(z^\frac{1}{p})$ in the right half-plane. By the Phragm\'en--Lindel\"of principle, assuming that $\varepsilon>0$ is sufficiently small, we have a bound $|F_\varepsilon(z)|=O(e^{a\pi|\Im z|})$ with $a=\frac{D(\Gamma,p)}{1+\delta}<D(\Gamma,p)$. 

By Lemma~\ref{lem:jensen-counting}, we know that
\[
\limsup_{x\to\infty}\frac{\log|F_\varepsilon(x)|}{x}\leq -2a\delta\sin(p\theta_0)<0
\]
independently of $\varepsilon>0$. Choosing $\varepsilon\in(0,2a\delta\sin(p\theta_0))$, we get a contradiction. This means that~\eqref{eq:indicator-lower-bound-contradiction} holds for no $\delta>0$ and $0<\theta_0<\frac{\pi}{2p}$, and consequently
\begin{equation}\label{eq:indicator-lower-bound}
D(\Gamma,p)\pi\sin(p\theta)+h_f(0)\cos(p\theta)\leq \max\{h_f(\theta),h_f(-\theta)\}, \quad\text{for all }\theta\in\left(0,\frac{\pi}{2p}\right).
\end{equation}
Clearly, this means that $\max\{h_f(\frac{\pi}{2p}),h_f(-\frac{\pi}{2p})\}\geq D(\Gamma,p)\pi$. Fix $\frac{\pi}{2p}<\theta<\frac{\pi}{p}$. If $h_f(\frac{\pi}{2p})\geq D(\Gamma,p)\pi$, then Theorem~\ref{thm:levin-indicator} applied in the sector $0<\arg z<\theta$ gives
\begin{align*}
    D(\Gamma,p)\pi\leq h_f\left(\frac{\pi}{2p}\right)&\leq\frac{\sin(p(\theta-\frac{\pi}{2p}))}{\sin(p(\theta-0))}h_f(0)+\frac{\sin(p(\frac{\pi}{2p}-0))}{\sin(p(\theta-0))}h_f(\theta)\\
    &=-\frac{\cos(p\theta)}{\sin(p\theta)}h_f(0)+\frac{1}{\sin(p\theta)}h_f(\theta).
\end{align*}
Rearranging, we get
\[
D(\Gamma,p)\pi\sin(p\theta)+h_f(0)\cos(p\theta)\leq h_f(\theta).
\]
If instead $h_f(-\frac{\pi}{2p})\geq D(\Gamma,p)\pi$, the same argument in the sector $-\theta<\arg z<0$ gives
\[
D(\Gamma,p)\pi\sin(p\theta)+h_f(0)\cos(p\theta)\leq h_f(-\theta).
\]
Combining this with~\eqref{eq:indicator-lower-bound} and the continuity of the indicator function, we get 
        \[
D(\Gamma,p)\pi\sin(p\theta)+h_f(0)\cos(p\theta)\leq \max\{h_f(\theta),h_f(-\theta)\}, \quad\text{for all }\theta\in\left[0,\pi/p\right].\qedhere
\]
\end{proof}
\section{Proof of Theorem~\ref{thm:uniqueness-gelfand-shilov}}
\begin{proof}[Proof of Theorem~\ref{thm:uniqueness-gelfand-shilov}]
We begin by proving part~(i) of the theorem. We first claim that, for $\tilde{a}\geq a>0$, $\tilde{b}\geq b>0$, and $\tilde{a}\tilde{b}<1$, the inequalities
\begin{equation}\label{eq:density-parameter-monotonicity}
    \sqrt{\tilde{a}\left(\frac{1}{\tilde{b}}-\tilde{a}\right)}\geq\sqrt{a\left(\frac{1}{b}-a\right)}\quad\text{and}\quad\sqrt{\tilde{b}\left(\frac{1}{\tilde{a}}-\tilde{b}\right)}\geq\sqrt{b\left(\frac{1}{a}-b\right)}
\end{equation}
can both hold only when $\tilde{a}=a$ and $\tilde{b}=b$. Indeed, assume that~\eqref{eq:density-parameter-monotonicity} holds. Then
\[
(1-\tilde{a}\tilde{b})^2=\tilde{a}\left(\frac{1}{\tilde{b}}-\tilde{a}\right)\tilde{b}\left(\frac{1}{\tilde{a}}-\tilde{b}\right)\geq a\left(\frac{1}{b}-a\right)b\left(\frac{1}{a}-b\right)=(1-ab)^2.
\]
This is only possible if $\tilde{a}\tilde{b}=ab$, but $\tilde{a}\geq a$ and $\tilde{b}\geq b$ so $\tilde{a}=a$ and $\tilde{b}=b$. 

This means that it is sufficient to prove the uniqueness result when
\[
\max\{h_f(0),h_f(\pi)\}=-a\pi\quad\text{and}\quad\max\{h_{\hat{f}}(0),h_{\hat{f}}(\pi)\}=-b\pi.
\]

Assume that $f|_\Lambda=0$, $\hat{f}|_M=0$, and $f\not\equiv 0$. We also assume that $h_f(0)\geq h_f(\pi)$. If not, we consider the function $g(x)=f(-x)$. Applying Lemma~\ref{lem:indicator-lower-bound-from-zeros} and recalling that $\max\{h_f(\theta),h_f(-\theta)\}\leq\frac{\pi}{b}\sin^2(\theta)$, we get
\[
D_\Lambda\pi \sin(2\theta)+h_{f}(0)\cos(2\theta)\leq \frac{\pi}{b}\sin^2\theta, \quad\text{for all }0\leq\theta\leq\frac{\pi}{2}.
\]
Since $D_\Lambda>\sqrt{a(\frac{1}{b}-a)}$, we know that
\[
\sqrt{a\left(\frac{1}{b}-a\right)}\sin(2\theta)-\frac{\sin^2\theta}{b}< a\cos(2\theta),\quad \text{for all }\theta\in\left(0,\frac{\pi}{2}\right).
\]
Choosing $\theta=\arctan\frac{\sqrt{a}}{\sqrt{\frac{1}{b}-a}}$, we get the contradiction $a<a$. This proves part~(i) of the theorem. 

We will now prove part~(ii) of the theorem. We can add an infinite $2$-smooth set $V=\{\lambda'_k\}_{k\in\mathbb N}$, disjoint from $\Lambda$, with sufficiently small density so that $\widetilde{\Lambda}:=\Lambda\cup V$ and $M$ still satisfy the density conditions of Lemma~\ref{lem:simultaneous-weighted-interpolation}.

For each $j\in\mathbb N$, let $\kappa_j=(\alpha_j,\beta_j)$ on $(\widetilde{\Lambda},M)$ be defined by
\[
\alpha_j(\lambda)=0\quad(\lambda\in\Lambda),\qquad
\alpha_j(\lambda'_k)=\delta_{j,k}\quad(k\in\mathbb N),\qquad
\beta_j(\mu)=0\quad(\mu\in M).
\]
Applying Lemma~\ref{lem:simultaneous-weighted-interpolation}, we obtain functions $\{f_j\}_{j\in\mathbb{N}}$ such that $f_j|_{\Lambda\backslash[-L,L]}=0$, $\hat{f}_j|_{M\backslash[-L,L]}=0$, and $f_j(\lambda'_k)=\delta_{j,k}$ whenever $\lambda'_k\notin[-L,L]$.

Since $\Lambda\cap[-L,L]$ and $M\cap[-L,L]$ only contain a finite number of points, we know by standard linear algebra that there exists a non-zero linear combination of these functions, denoted by $f$, such that $f|_\Lambda=0$ and $\hat{f}|_M=0$.
\end{proof}
\section{Proofs of the Pauli results}
\subsection{Proof of the one-sided modulus theorem}

In the proofs below we regularize the sampling sets: under a strict $\limsup$ spacing hypothesis one may thin each half-line to a $2$-smooth subset with density still above the corresponding threshold, while under a strict $\liminf$ spacing hypothesis one may augment each half-line to a $2$-smooth superset with density still below the corresponding threshold. Passing to a subset is harmless in uniqueness arguments, and passing to a superset is harmless in non-uniqueness constructions.

\begin{proof}[Proof of Theorem~\ref{thm:pauli-sharpening}]
We begin by proving part~(i) of the theorem. By thinning $\Lambda$ on each half-line, it is enough to prove the result when $\Lambda_\pm$ are $2$-smooth and $D(\Lambda_\pm,2)>\frac{\mathfrak{c}_1(A)}{2}$. 

Assume that $|f(x)|\not\equiv|g(x)|$ on $\mathbb{R}$. Let 
$F(z):=f(z)\overline{f(\overline{z})}$ and $G(z):=g(z)\overline{g(\overline{z})}$. Clearly, $h_{F-G}(0)\leq -2A\pi$, and since $F(z)-G(z)\not\equiv 0$, we know that $h_{F-G}(0)>-\infty$.

By Fourier inversion, we see that $h_f(\theta),h_g(\theta)\leq\frac{\pi}{A}\sin^2\theta$. Applying Lemma~\ref{lem:indicator-lower-bound-from-zeros} and
\[
\max\{h_{F-G}(\theta),h_{F-G}(-\theta)\}\leq\frac{2\pi}{A}\sin^2\theta,
\]
we see that
\[
D(\Lambda_+,2)\pi\sin(2\theta)+h_{F-G}(0)\cos(2\theta)\leq \frac{2\pi}{A}\sin^2\theta, \quad\text{for all }\theta\in\left[0,\pi/2\right].
\]
Since $h_{F-G}(0)\leq -2A\pi$, we can maximize $D(\Lambda_+,2)$ and get
\[
D(\Lambda_+,2)\leq \begin{cases}
    \frac{1}{A},\quad 0<A<\frac{1}{\sqrt{2}}\\
    2\sqrt{1-A^2},\quad\frac{1}{\sqrt{2}}\leq A< 1.
\end{cases}
\]
	Since $D(\Lambda_+,2)>\frac{\mathfrak{c}_1(A)}{2}$ this is a contradiction. Hence $F-G\equiv 0$, so $|f(x)|=|g(x)|$ for all $x\in\mathbb{R}$.

	We now prove part~(ii) of the theorem. By augmenting $\Lambda$ on each half-line, it is enough to prove the result when $\Lambda_\pm$ are $2$-smooth and $D(\Lambda_\pm,2)<\frac{\mathfrak{c}_1(A)}{2}$. We split this regularized set into two sets $\Lambda_1:=\{\lambda_{2j}\}_{j\in\mathbb{Z}}$ and $\Lambda_2:=\{\lambda_{2j+1}\}_{j\in\mathbb{Z}}$. Then
\[
\max\{D(\Lambda_1,2),D(\Lambda_2,2)\}<\begin{cases}
    \frac{1}{2A},\quad 0<A<\frac{1}{\sqrt{2}}\\
    \sqrt{1-A^2},\quad\frac{1}{\sqrt{2}}\leq A< 1.
\end{cases}
\]
Hence, by the non-uniqueness part of Theorem~\ref{thm:uniqueness-gelfand-shilov}, applied with an auxiliary frequency set of arbitrarily small density and with parameters $(a,b)=(\frac{1}{2A},A)$ when $0<A<\frac{1}{\sqrt{2}}$ and $(a,b)=(A,A)$ when $\frac{1}{\sqrt{2}}\leq A<1$, there exist non-zero functions $\Phi,\Psi\in E(A,A)$ such that $\Phi|_{\Lambda_1}=0$ and $\Psi|_{\Lambda_2}=0$. Since $\Phi(x)\overline{\Psi(x)}$ is not identically zero on $\mathbb R$, we may replace $\Psi$ by $e^{i\vartheta}\Psi$ for a suitable $\vartheta\in\mathbb R$ and assume that
\[
|\Phi(x)+\Psi(x)|\not\equiv|\Phi(x)-\Psi(x)|.
\]
We now let $f=\Phi+\Psi$ and $g=\Phi-\Psi$. Clearly, $|f(\lambda_{2j})|=|\Psi(\lambda_{2j})|=|g(\lambda_{2j})|$ for all $j\in\mathbb{Z}$ and $|f(\lambda_{2j+1})|=|\Phi(\lambda_{2j+1})|=|g(\lambda_{2j+1})|$ for all $j\in\mathbb{Z}$. This completes the proof of part~(ii) of the theorem.
\end{proof}
\subsection{Frequency-matching construction}
The proof of Theorem~\ref{thm:discrete-pauli-frequency-match-nonuniqueness} is quite similar to the proof of the second part of Theorem~\ref{thm:pauli-sharpening}, but we need to make sure that $\hat{\Phi}(\xi)$ and $\hat{\Psi}(\xi)$ are orthogonal for all $\xi\in\mathbb{R}$. To construct these $\Phi$ and $\Psi$, we will use the following result about canonical products.
\begin{theoremx}[Theorem~2, Lecture 12 in~\cite{levinentirefunctions}]\label{thm:canonical-product-estimate}
    If $\Gamma:=(\gamma_n)_{n=0}^\infty$ is a sequence of positive numbers such that the limit $\displaystyle{\Delta:=\lim_{n\to\infty}\frac{n}{\gamma_n}}$ exists, and if
    \[
    \Pi(z):=\prod_{n=0}^\infty\left(1-\frac{z^2}{\gamma_n^2}\right),
    \]
    then $h_\Pi(\theta)=\Delta\pi|\sin\theta|$.
\end{theoremx}
We are now ready to prove Theorem~\ref{thm:discrete-pauli-frequency-match-nonuniqueness}.
\begin{proof}[Proof of Theorem~\ref{thm:discrete-pauli-frequency-match-nonuniqueness}]
    If $A\geq\frac{\sqrt{3}}{2}$, then $\mathfrak{c}_1(A)\leq 2$, hence by the non-uniqueness part of Theorem~\ref{thm:kns-uniqueness}, we can pick $g(x):=0$ and a non-trivial function $f\in \mathcal{H}$ such that $f|_\Lambda=0$ and $\hat{f}|_M=0$. Thus we can assume $0<A<\frac{\sqrt{3}}{2}$. By augmenting $\Lambda$, it is enough to construct the pair for a $2$-smooth superset with density $D(\Lambda,2)<\frac{\max\{\mathfrak{c}_1(A),2\}}{2}$. 
    
    Let
    \[
    \Pi_1(z):=\prod_{n=0}^\infty\left(1-\frac{z^2}{\lambda_{2n}^4}\right)\quad\text{and}\quad\Pi_2(z):=\prod_{n=0}^\infty\left(1-\frac{z^2}{\lambda_{2n+1}^4}\right).
    \]
    By Theorem~\ref{thm:canonical-product-estimate}, we know that $\Pi_1$ and $\Pi_2$ converge and have Phragm\'en--Lindel\"of indicator
    \[
    h_{\Pi_1}(\theta),h_{\Pi_2}(\theta)=\frac{D(\Lambda,2)}{2}\pi|\sin(\theta)|.
    \]
    We can now define the following functions
    \[
    \Phi(z):=e^{-(\sigma_A+\varepsilon)\pi z^2}\Pi_1(z^2)\quad\text{and}\quad \Psi(z):=e^{-(\sigma_A+\varepsilon)\pi z^2}z\Pi_2(z^2)
    \]
    where 
    \[
    \sigma_A:=\begin{cases}
        \frac{1}{2A},\quad&0<A<\frac{1}{\sqrt{2}},\\
        A,\quad& \frac{1}{\sqrt{2}}\leq A<\frac{\sqrt{3}}{2}
    \end{cases}
    \]
    and $\varepsilon>0$ is chosen sufficiently small such that
    \[
    \frac{D(\Lambda,2)}{2}<\sqrt{(\sigma_A+\varepsilon)\left(\frac{1}{A}-(\sigma_A+\varepsilon)\right)}.
    \]
    Both $\Phi$ and $\Psi$ have Phragm\'en--Lindel\"of indicator
    \[
    h_\Phi(\theta),h_\Psi(\theta)=-\pi(\sigma_A+\varepsilon)\cos(2\theta)+\frac{D(\Lambda,2)}{2}\pi|\sin(2\theta)|.
    \]
    Clearly $\sigma_A+\varepsilon>A$ and by Lemma~\ref{lem:fourier-decay-bound}, we know that $\max\{h_{\hat{\Phi}}(0),h_{\hat{\Phi}}(\pi),h_{\hat{\Psi}}(0),h_{\hat{\Psi}}(\pi)\}<-A\pi$, so $\Phi,\Psi\in E(A,A)$. 
    
    Consider now the functions 
    \[
    f(z):=\Phi(z)+\Psi(z)\quad\text{and}\quad g(z):=\Phi(z)-\Psi(z).
    \]
    Clearly, $f,g\in E(A,A)$ since $\Phi,\Psi\in E(A,A)$.
    
    Since $\Phi(x)$ is even and real-valued, we know that $\hat{\Phi}(\xi)$ is real-valued. Similarly, since $\Psi(x)$ is odd and real-valued, we know that $\hat{\Psi}(\xi)$ is imaginary-valued. Because of this pointwise orthogonality, we know that
    \[
    |\hat{f}(\xi)|^2=|\hat{\Phi}(\xi)+\hat{\Psi}(\xi)|^2=|\hat{\Phi}(\xi)-\hat{\Psi}(\xi)|^2=|\hat{g}(\xi)|^2\quad\text{for all }\xi\in\mathbb{R}.
    \]
    Additionally, for all $\lambda\in\Lambda$, we have $\Phi(\pm\lambda)=0$ or $\Psi(\pm\lambda)=0$. Hence, for all $\lambda\in\Lambda$, we have 
    \[
|f(\pm\lambda)|=|\Phi(\pm\lambda)|=|g(\pm\lambda)|\quad\text{or}\quad|f(\pm\lambda)|=|\Psi(\pm\lambda)|=|g(\pm\lambda)|.
    \]
    This means that $|f(\pm\lambda)|=|g(\pm\lambda)|$ for all $\lambda\in\Lambda$. Since both $\Phi$ and $\Psi$ are real-valued and non-zero, we also know that $|f(x)|\not\equiv|g(x)|$.
\end{proof}

\subsection{Proof of the weak Pauli pair theorem}
We will prove the theorem in the following way. First, the sampled equality of moduli gives zeros of the two entire functions $H$ and $\tilde{H}$. Second, indicator lower bounds convert these zeros into density inequalities. Finally, optimizing those inequalities over the possible decay parameters gives $\mathfrak{c}_2(A)$.

\begin{proof}[Proof of Theorem~\ref{thm:pauli-non-uniqueness}]
We begin by proving part~(i) of the theorem. By thinning $\Lambda$ and $M$ on each half-line, it is enough to prove the result when $\Lambda_+,\Lambda_-,M_+,M_-$ are $2$-smooth and have densities larger than $\frac{\mathfrak{c}_2(A)}{2}$. By Theorem~\ref{thm:discrete-beurling}, we know that $h_f(0),h_f(\pi),h_{\hat{f}}(0),h_{\hat{f}}(\pi)\leq-A\pi$. Let
\[
u(z):=f(z)-g(z)\quad\text{and}\quad v(z):=f(z)+g(z).
\]
If $u\equiv0$ or $v\equiv0$, then $f$ and $g$ form a Pauli pair. Hence we may assume that both $u$ and $v$ are non-zero.
We now define the entire function
\[
H(z):=\frac{1}{2}\left(u(z)\overline{v(\bar z)}+\overline{u(\bar z)}v(z)\right)=f(z)\overline{f(\bar z)}-g(z)\overline{g(\bar z)}.
\]
We now let $a_1:=-\pi^{-1}\max\{h_u(0),h_u(\pi)\}$ and $b_1:=-\pi^{-1}\max\{h_{\hat{u}}(0),h_{\hat u}(\pi)\}$. Similarly, $a_2:=-\pi^{-1}\max\{h_v(0),h_v(\pi)\}$ and $b_2:=-\pi^{-1}\max\{h_{\hat v}(0),h_{\hat v}(\pi)\}$. Then
\[
h_H(\theta)\leq\max\{h_u(\theta)+h_v(-\theta),h_u(-\theta)+h_v(\theta)\}\quad\text{for all $\theta$.}
\]
Hence $h_H(0)\leq -(a_1+a_2)\pi$ and
\[
D(\Lambda_+,2)\pi\sin(2\theta)+h_H(0)\cos(2\theta)\leq\left(\frac{1}{b_1}+\frac{1}{b_2}\right)\pi\sin^2\theta\quad\text{for all }\theta\in[0,\pi/2].
\]
Doing the same optimization as in the proof of Theorem~\ref{thm:uniqueness-gelfand-shilov}(i), we have
\begin{equation}\label{eq:time-density-bound}
D(\Lambda_+,2)\leq\sqrt{(a_1+a_2)\left(\frac{1}{b_1}+\frac{1}{b_2}-(a_1+a_2)\right)}
\end{equation}
if $H\not\equiv0$. Similarly, we can define the entire function
\[
\tilde{H}(z):=\frac{1}{2}\left(\hat{u}(z)\overline{\hat{v}(\bar z)}+\overline{\hat{u}(\bar z)}\hat{v}(z)\right)=\hat{f}(z)\overline{\hat{f}(\bar z)}-\hat{g}(z)\overline{\hat{g}(\bar z)}
\]
and by the same argument we have
\begin{equation}\label{eq:frequency-density-bound}
D(M_+,2)\leq\sqrt{(b_1+b_2)\left(\frac{1}{a_1}+\frac{1}{a_2}-(b_1+b_2)\right)}
\end{equation}
if $\tilde{H}\not\equiv0$. 
Let $S$ and $T$ denote the right-hand side of~\eqref{eq:time-density-bound} and~\eqref{eq:frequency-density-bound} respectively. Assuming $H\not\equiv0$ and $\tilde{H}\not\equiv 0$, we have $\left(\frac{\mathfrak{c}_2(A)}{2}\right)^2<\min\{S^2,T^2\}$. 

We now wish to upper bound $B:=\min\{S^2,T^2\}$. Let $s:=a_1+a_2$, $t:=b_1+b_2$, $\eta:=a_1a_2$, $\nu:=b_1 b_2$, and $x:=\sqrt{\eta\nu}$. Plugging these into the definition of $S$ and $T$, we have
\begin{equation}\label{eq:st-density-parameterization}
    S^2=s\left(\frac{t}{\nu}-s\right)\quad\text{and}\quad T^2=t\left(\frac{s}{\eta}-t\right).
\end{equation}
From~\eqref{eq:st-density-parameterization}, we get $\nu\left(s^2+B\right)\leq st$ and $\eta\left(t^2+B\right)\leq st.$
Adding these together and applying Young's inequality, we see that
\begin{align*}
    2st\geq (\nu+\eta)B+\nu s^2+\eta t^2
    &\geq (\nu+\eta)B+2st\sqrt{\nu\eta}= (\nu+\eta)B+2stx.
\end{align*}
Rearranging, we get $B\leq \frac{2st(1-x)}{\nu+\eta}$. We now want to show that
\[
\frac{2st}{\nu+\eta}\leq\left(\frac{A}{\sqrt{x}}+\frac{\sqrt{x}}{A}\right)^2.
\]
Assume without loss of generality that $a_1\leq a_2$. Then since $a_1+\frac{\eta}{a_1}$ is non-increasing for all $a_1\in[A,\sqrt{\eta}]$, we have $s=a_1+a_2=a_1+\frac{\eta}{a_1}\leq A+\frac{\eta}{A}$. Similarly, we also have $t\leq A+\frac{\nu}{A}$. This means that
\begin{align*}
    \frac{2st}{\eta+\nu}&\leq\frac{2(A+\frac{\eta}{A})(A+\frac{\nu}{A})}{\nu+\eta}=2+\frac{2A^2+\frac{2x^2}{A^2}}{\nu+\eta}.
\end{align*}
Since the arithmetic mean of two non-negative numbers is always greater than their geometric mean, we know that $\nu+\eta\geq 2\sqrt{\eta\nu}=2x$. Hence,
\begin{align*}
    2+\frac{2A^2+\frac{2x^2}{A^2}}{\nu+\eta}\leq 2+\frac{A^2}{x}+\frac{x}{A^2}=\left(\frac{A}{\sqrt{x}}+\frac{\sqrt{x}}{A}\right)^2.
\end{align*}
Consequently, we know that
\begin{equation}\label{eq:case-one-density-upper-bound}
    \left(\frac{\mathfrak{c}_2(A)}{2}\right)^2<B\leq \frac{2st(1-x)}{\nu+\eta}\leq(1-x)\left(\frac{A}{\sqrt{x}}+\frac{\sqrt{x}}{A}\right)^2.
\end{equation}
Since $x\in[A^2,1]$, we know that the right-hand side of~\eqref{eq:case-one-density-upper-bound} is maximized when $x=x_A$ where
\[
x_A:=\begin{cases}
    \frac{1+\sqrt{1-8A^2}}{4},\quad &0<A<\frac{1}{3},\\
    A^2,\quad &\frac{1}{3}\leq A<1.
\end{cases}
\]
Plugging this into~\eqref{eq:case-one-density-upper-bound}, we get $\left(\frac{\mathfrak{c}_2(A)}{2}\right)^2<\left(\frac{\mathfrak{c}_2(A)}{2}\right)^2$ which is a contradiction. Hence, $H\equiv0$ or $\tilde{H}\equiv 0$.

We now move on to part~(ii) of the theorem. By augmenting $\Lambda$ and $M$ on each half-line, it is enough to construct the pair when $\Lambda_+,\Lambda_-,M_+,M_-$ are $2$-smooth and have densities smaller than $\frac{\mathfrak{c}_2(A)}{2}$. If $A\geq\frac{\sqrt{3}}{2}$, then $\mathfrak{c}_2(A)=2$, hence by the non-uniqueness part of Theorem~\ref{thm:kns-uniqueness}, we can pick $g(x):=0$ and a non-trivial function $f\in \mathcal{H}$ such that $f|_\Lambda=0$ and $\hat{f}|_M=0$. Thus we can assume $0<A<\frac{\sqrt{3}}{2}$. 

We can now choose $a_1,b_2=A$ and $a_2,b_2=
\frac{x_A}{A}$. This means that
\begin{align*}
    \sqrt{a_1\left(\frac{1}{b_1}-a_1\right)}+\sqrt{a_2\left(\frac{1}{b_2}-a_2\right)}&=\sqrt{b_1\left(\frac{1}{a_1}-b_1\right)}+\sqrt{b_2\left(\frac{1}{a_2}-b_2\right)}\\
    &=\begin{cases}
            \sqrt{\frac{(3-\sqrt{1-8A^2})^3}{8(1-\sqrt{1-8A^2})}},\quad &0<A<\frac{1}{3}\\
            2\sqrt{1-A^2},\quad &\frac{1}{3}\leq A<\frac{\sqrt{3}}{2}.
    \end{cases}
\end{align*}
We now know by the non-uniqueness part of Theorem~\ref{thm:uniqueness-gelfand-shilov} that there exist $\Lambda_1,M_1,\Lambda_2,M_2\subset\mathbb{R}$ such that $\Lambda=\Lambda_1\cup\Lambda_2$, $M=M_1\cup M_2$, and there exist non-zero functions $\Phi,\Psi\in E(A,A)$ such that $\Phi|_{\Lambda_1}=0$, $\hat{\Phi}|_{M_1}=0$, $\Psi|_{\Lambda_2}=0$ and $\hat{\Psi}|_{M_2}=0$.

Since $\Phi(x)\overline{\Psi(x)}$ and $\hat{\Phi}(\xi)\overline{\hat{\Psi}(\xi)}$ are not identically zero on $\mathbb R$, we may replace $\Psi$ by $e^{i\vartheta}\Psi$ for a suitable $\vartheta\in\mathbb R$ and assume that both
\[
|\Phi(x)+\Psi(x)|\not\equiv|\Phi(x)-\Psi(x)|\quad\text{and}\quad|\hat{\Phi}(\xi)+\hat{\Psi}(\xi)|\not\equiv|\hat{\Phi}(\xi)-\hat{\Psi}(\xi)|,
\]
hold. We now let $f=\Phi+\Psi$ and $g=\Phi-\Psi$. Clearly, $|f(\lambda)|=|g(\lambda)|$ for all $\lambda\in\Lambda$ and $|\hat{f}(\mu)|=|\hat{g}(\mu)|$ for all $\mu\in M$, while $|f(x)|\not\equiv|g(x)|$ and $|\hat{f}(\xi)|\not\equiv|\hat{g}(\xi)|$. This completes the proof of part~(ii) of the theorem.
\end{proof}

\section*{Acknowledgments}
The author thanks Kristian Seip for helpful comments and feedback on the exposition.

\bibliographystyle{plain}
\bibliography{references}
\end{document}